\documentclass[11pt,english]{article}
\usepackage{amsfonts,amsmath,amsthm,amscd,amssymb,latexsym,cite,verbatim,texdraw,floatflt,
caption2,pb-diagram}

\usepackage{tikz} 
\usetikzlibrary{positioning,trees, decorations.markings}

\theoremstyle{plain}
\newtheorem{theorem}{Theorem}[section]
\newtheorem{lemma}{Lemma}[section]
\newtheorem{proposition}{Proposition}[section]
\newtheorem{corollary}{Corollary}[section]
\newtheorem{definition}{Definition}[section]
\theoremstyle{definition}

\newtheorem{remark}{Remark}[section]
\newcommand{\keywords}{\textbf{Key words. }\medskip}
\newcommand{\subjclass}{\textbf{MSC 2020. }\medskip}
\renewcommand{\abstract}{\textbf{Abstract. }\medskip}

\numberwithin{equation}{section}

\DeclareMathOperator{\diam}{diam}
\DeclareMathOperator{\Iso}{Iso}
\DeclareMathOperator{\Fix}{Fix}
\newcommand{\mfu}{\mathfrak{U}}
\newcommand{\mct}{\mathfrak{T}}
\newcommand{\RR}{\mathbb{R}}
\newcommand{\abs}[1]{\left\vert#1\right\vert}
\newcommand{\Sp}[1]{\operatorname{Sp}(#1)}
\def\we{\mathrel{\stackrel{\rm w}=}}

\begin{document}

\title{Hereditary properties of finite  ultrametric spaces}

\author{Evgeniy A. Petrov}

\maketitle

\begin{abstract}
A characterization of finite homogeneous ultrametric spaces and finite ultrametric spaces generated by unrooted labeled trees is found in terms of representing trees. A characterization of finite ultrametric spaces having perfect strictly $n$-ary trees is found in terms of special graphs connected with the space. Further, we give a detailed survey of some special classes of finite ultrametric spaces, which were considered in the last ten years, and study their hereditary properties. More precisely, we are interested in the following question. Let $X$ be an arbitrary finite ultrametric space from some given class. Does every subspace of $X$ also belong to this class?
\end{abstract}



\section{Introduction}

In 2001 at the Workshop on General Algebra~\cite{WGA} the attention of experts on the theory of lattices was paid to the following problem of I.~M.~Gelfand: using graph theory describe up to isometry all finite ultrametric spaces. An appropriate representation of ultrametric spaces $X$ by monotone rooted trees $T_X$ was proposed in~\cite{GurVyal(2012)} that can be considered in some sense as a solution of above mentioned problem. The question naturally arises about applications of this representation. One such application is the structural characteristic of finite ultrametric spaces for which the Gomory-Hu inequality becomes an equality, see~\cite{PD(UMB)}. The ultrametric spaces for which its representing trees are strictly binary were described in~\cite{DPT}. A characterization of finite ultrametric spaces which are as rigid as possible was also obtained, see~\cite{DPT(Howrigid)}. Extremal properties of finite ultrametric spaces and related them properties of monotone rooted trees have been found in~\cite{DP20}. Se also papers~\cite{DP19, DP18, Do19, Do20BBMS, Do20, P18} for some another properties of ultrametric spaces based on analysis of representing trees. The present paper is also a contribution to this line of studies.

The paper is organized as follows.
The first section of the paper contains the main definitions and the required technical results.
In Section 2 and Section 4 finite homogenous ultrametric spaces and finite ultrametric spaces defined by unrooted labeled trees are described in terms of representing trees.
In Section 3 we give a characterizations of finite ultrametric spaces having perfect strictly $n$-ary trees in terms of special graphs $G_{r,X}$ connected with the space $X$.
In Section 5 we give a detailed survey of some special classes of finite ultrametric spaces which were considered in the last ten years. In Section 6 from the above mentioned classes we distinguish the classes such that every subspace of a space from a given class also belongs to this class.

Recall some definitions from the theory of metric spaces and the graph theory.
An \textit{ultrametric} on a set $X$ is a function $d\colon X\times X\rightarrow \mathbb{R}^+$, $\mathbb R^+ = [0,\infty)$, such that for all $x,y,z \in X$:
\begin{itemize}
\item [\textup{(i)}] $d(x,y)=d(y,x)$,
\item [\textup{(ii)}] $(d(x,y)=0)\Leftrightarrow (x=y)$,
\item [\textup{(iii)}] $d(x,y)\leq \max \{d(x,z),d(z,y)\}$.
\end{itemize}
Inequality (iii)  is often called the {\it strong triangle inequality}.
The pair $(X,d)$ is called an \emph{ultrametric space.}
The \emph{spectrum} of an ultrametric space $(X,d)$ is the set  $$\operatorname{Sp}(X)=\{d(x,y)\colon x,y \in  X\}.$$
For simplicity we will always assume that $X\cap \Sp{X}=\varnothing$.
The spectrum $\operatorname{Spec}(X,x)$ of the space $X$ at the point $x$ is the set
$$
\operatorname{Spec}(X,x)= \{d(x,y)\, | \, y\in X \}.
$$
The quantity
$$
\diam X=\sup\{d(x,y)\colon x,y\in X\}.
$$
is the \emph{diameter} of the space $(X,d)$.

Recall that a \textit{graph} is a pair $(V,E)$ consisting of a nonempty set $V$ and a (probably empty) set $E$  elements of which are unordered pairs of different points from $V$. For a graph $G=(V,E)$, the sets $V=V(G)$ and $E=E(G)$ are called \textit{the set of vertices} and \textit{the set of edges}, respectively. Recall that a \emph{path} is a nonempty graph $P=(V,E)$ of the form
$$
V=\{x_0,x_1,...,x_k\}, \quad E=\{\{x_0,x_1\},...,\{x_{k-1},x_k\}\},
$$
where all $x_i$ are different.
A connected graph without cycles is called a \emph{tree}. A tree $T$ may have a distinguished vertex called the \emph{root}; in this case $T$ is called a \emph{rooted tree}.  An \emph{$n$-ary tree} is a rooted tree, such that the degree of each of its vertices is at most $n+1$.
A rooted tree is \emph{strictly $n$-ary} if every internal node has exactly $n$ children. In the case $n=2$ such tree is called \emph{strictly binary}.
Generally we  follow terminology used in~\cite{BM}.

Let $k\geqslant 2$. A nonempty graph $G$ is called \emph{complete $k$-partite} if its vertices can be divided into $k$ disjoint nonempty sets $X_1,...,X_k$ so that there are no edges joining the vertices of the same set $X_i$ and any two vertices from different $X_i,X_j$, $1\leqslant i,j \leqslant k$ are adjacent. In this case we write $G=G[X_1,...,X_k]$.
We shall say that $G$ is a {\it complete multipartite graph} if there exists
 $k \geqslant 2$ such that $G$ is complete $k$-partite, cf. \cite{Di}.

\begin{definition}[\!\!\cite{DDP(P-adic)}]\label{d2}
Let $(X,d)$ be a finite ultrametric space. Define the graph $G_X^d$ as follows $V(G_X^d)=X$ and
$$
(\{u,v\}\in E(G_X^d))\Leftrightarrow(d(u,v)=\diam X).
$$
We call $G_X^d$ the \emph{diametrical graph} of $X$.
\end{definition}
\begin{definition}\label{d14}
Let $(X,d)$ be an ultrametric space with $|X|\geqslant 2$ and  the spectrum $\operatorname{Sp}(X)$ and let $r\in \operatorname{Sp}(X)$ be nonzero. Define by $G_{r,X}$ a graph for which $V(G_{r,X})=X$ and
$$
(\{u,v\}\in E(G_{r,X}))\Leftrightarrow (d(u,v)=r).
$$

For $r=\operatorname{diam} X$ it is clear that $G_{r,X}$ is the diametrical graph of $X$.
\end{definition}

\begin{theorem}[\!\!\cite{DDP(P-adic)}]\label{t13}
Let $(X,d)$ be a finite ultrametric space, $|X|\geqslant 2$. Then $G_X^d$ is complete multipartite.
\end{theorem}

With every finite ultrametric space $(X, d)$, we can associate  a labeled rooted $n$-ary tree $T_X$ by the following rule (see~\cite{PD(UMB)}). If $X=\{x\}$ is a one-point set, then $T_X$ is the rooted tree consisting of one node $X$ labeled by $0$. Let $|X|\geqslant 2$.
According to Theorem~\ref{t13} we have $G^d_X = G^d_X[X_1,...,X_k]$.
In this case the root $X$ of the tree $T_X$ is labeled by $\diam X$ and, moreover, $T_X$ has $k$ nodes $X_1,...,X_k$ of the first level with the labels
\begin{equation}\label{e2.7}
l(X_i)=
\diam X_i, \quad i = 1,...,k.
\end{equation}
The nodes of the first level indicated by $0$ are leaves, and those indicated by strictly positive numbers are internal nodes of the tree $T_X$. If the first level has no internal nodes, then the tree $T_X$ is constructed. Otherwise, by repeating the above-described procedure with $X_i$ corresponding to the internal nodes of the first level, we obtain the nodes of the second level, etc. Since $|X|$ is finite, all vertices on some level will be leaves and the construction of $T_X$ is completed.
The above-constructed labeled tree $T_X$ is called the \emph{representing tree} of the space $(X, d)$.
To underline that $l_X(v)$, $v\in V(T_X)$, is a labeling function of the representing tree $T_X$ we shall write $(T_X,l_X)$.
The rooted tree $T_X$ without the labels we will denote by $\overline{T}_X$.

Let $T$ be a rooted tree. For every vertex $v$ of $T$ we denote by $T_v$ the subtree of $T$ such that \(v\) is the root of \(T_v\),
\begin{equation*}\label{e2.5}
V(T_v) = \{u \in V(T) \colon u = v \text{ or \(u\) is a successor of } v \text{ in } T\},
\end{equation*}
and satisfying
\begin{equation*}\label{e2.6}
(\{v_1, v_2\} \in E(T_v)) \Leftrightarrow (\{v_1, v_2\} \in E(T))
\end{equation*}
for all \(v_1\), \(v_2 \in V(T_v)\).
Denote by \(\overline{L}(T_v)\) the set of all leaves of \(T_v\). If \(T = T_X\) is the representing tree of a finite ultrametric space \((X, d)\), \(v \in V(T_X)\) and $\overline{L}(T_v) = \{\{x_1\}, \ldots, \{x_n\}\}$, then for simplicity we write $L(T_v) = \{x_1, \ldots, x_n\}$. Consequently, the equality $v = L(T_v)$ holds for every \(v \in V(T_X)\).

Let  $v$ be a node of the rooted tree $T$. Denote by $\delta^+(v)$ the \emph{out-degree} of $v$, i.e., $\delta^+(v)$ is the number of children of $v$, and by $\operatorname{lev}(v)$ denote the level of a node $v \in V(T)$.

Note that the correspondence between ultrametric spaces and trees or tree-like structures is well known, cf.~\cite{GV, GNS00,GurVyal(2012),H04,Le03,Ho01,BS17}.

\begin{definition}\label{d15}
Let $G=(V,E)$ be nonempty graph, and let $V_0$ be the set (possibly empty) of all isolated vertices of the graph $G$. Denote by $G'$ the subgraph of the graph $G$, generated by the set $V\backslash V_0$.
\end{definition}

\begin{lemma}[\!\!\cite{DP19}]\label{c25}
Let $(X, d)$ be a finite ultrametric space with $|X|\geqslant 2$ and let $r \in \Sp{X}\setminus \{0\}$. Then the graph $G'_{r, X}$ is a union of $p$ complete multipartite graphs $G^1_{r}, \ldots, G^p_{r}$, where $p$ is the number of all distinct nodes \(x_1\), \(\ldots\), \(x_p\) of \(T_X\) labeled by \(r\) and this union is disjoint if \(p \geqslant 2\). Moreover, for every \(i \in \{1, \ldots, p\}\), the graph \(G_{r}^i\) is complete \(k\)-partite,
\begin{equation}\label{eq25}
G_{r}^i = G_{r}^i [L_{T_{x_{i1}}}, \ldots, L_{T_{x_{ik}}}],
\end{equation}
where \(k = \delta^{+}(x_i)\) and $x_{i1}, \ldots, x_{ik}$ are the direct successors of $x_i$.
\end{lemma}

Let $(X,d)$ be an ultrametric space. Recall that a \emph{ball} with a radius $r \geqslant 0$ and a center $c\in X$ is the set
\[
B_r(c)=\{x\in X\colon d(x,c)\leqslant r\}.
\]
The \emph{ballean} $\mathbf{B}_X$ of the ultrametric space $(X,d)$ is the set of all balls of $(X,d)$. Every one-point subset of $X$ belongs to $\mathbf{B}_X$,  this is a \emph{singular} ball in~$X$.

The following proposition claims that the ballean of a finite ultrametric space $(X,d)$ is the vertex set of the representing tree $T_X$.

\begin{proposition}[\!\!\cite{P(TIAMM)}]\label{lbpm}
Let $(X,d)$ be a finite nonempty ultrametric space with the representing tree $T_X$. Then the following statements hold.
\begin{itemize}
\item [\textup{(i)}] $L_{T_v}$ belongs to $\mathbf{B}_X$ for every node $v\in V(T_X)$.
\item [\textup{(ii)}] For every $B \in \mathbf{B}_X$ there exists the unique node $v$ such that $L_{T_v}=B$.
\end{itemize}
\end{proposition}

\begin{theorem}[\!\!\cite{DP18}]\label{t2.9}
Let \((X, d)\) be a finite ultrametric space and let \(x_1\) and \(x_2\) be two different points of \(X\). If \(P\) is the path joining the different leaves \(\{x_1\}\) and \(\{x_2\}\) in \((T_X,l_X)\), then we have
\[
d(x_1, x_2) = \max_{v \in V(P)} l_X(v).
\]
\end{theorem}

Recall that metric spaces $(X,d)$ and $(Y, \rho)$ are \emph{isometric} if there is a bijection $f\colon X\to Y$ such that the equality
$$
d(x,y) = \rho(f(x),f(y))
$$
holds for all $x$, $y \in X$.

\begin{definition}
Let $T_1$ and $T_2$ be rooted trees with the roots $v_1$ and $v_2$, respectively. A bijective function $\Psi\colon V(T_1)\to V(T_2)$ is an isomorphism of $T_1$ and $T_2$ if
$$
(\{x,y\}\in E(T_1))\Leftrightarrow(\{\Psi(x),\Psi(y)\}\in E(T_2))
$$
for all $x,y \in V(T_1)$ and $\Psi(v_1)=v_2$. If there exists an isomorphism of rooted trees $T_1$ and $T_2$, then we will write $T_1\simeq T_2$.
\end{definition}

We shall say that trees $(T_X,l_X)$ and $(T_Y,l_Y)$ are isomorphic as labeled rooted trees if  $\overline{T}_X \simeq \overline{T}_Y$ with an isomorphism $\Psi$ and $l_X(x)=l_Y(\Psi(x))$, $x\in V(T_X)$.

\begin{theorem}[\!\!\cite{DPT(Howrigid)}]\label{l3.3}
Let $(X, d)$ and $(Y, \rho)$ be nonempty finite ultrametric spaces. Then the representing trees $T_X$ and $T_Y$ are isomorphic as labeled rooted trees if and only if $(X, d)$ and $(Y, \rho)$ are isometric.
\end{theorem}

\section{Finite homogeneous ultrametric spaces}

A relational structure $\textbf R$ is homogeneous if every isomorphism between finite induced
substructures of $\textbf R$ extends to an automorphism of the whole structure $\textbf R$ itself. Introduced by Fra\"{\i}ss\'{e}~\cite{Fr54} and J\'{o}nsson ~\cite{Jo60}, homogeneous structures are now playing a fundamental role in model theory. According to the terminology of Fra\"{\i}ss\'{e}~\cite{Fr00}, a metric space $X$ is homogeneous if every isometry $f$ whose domain and range are finite subsets of $X$ extends to surjective isometry of $X$ onto $X$.
Some results related to indivisible homogeneous ultrametric spaces can be found in~\cite{DLPS07, DLPS08}.
The more deep investigation of homogeneous ultrametric spaces was done in~\cite{DLPS16}, where in particular the following result was obtained.

\begin{theorem}[\!\!\cite{DLPS16}]\label{t31}
An ultrametric space $X$ is homogeneous if and only if the following two conditions hold:
\begin{itemize}
  \item [\textup{(1)}] $\operatorname{Spec}(X,x)= \operatorname{Spec}(X,x')$ for all $x, x' \in X$.
  \item [\textup{(2)}] Balls of $X$ with the same kind are isometric.
\end{itemize}
\end{theorem}

Two balls have the same kind if they have the same diameter which is attained in both or in none. Clearly, for finite metric spaces this definition can be reduced to the condition that balls have the same diameter since in this case diameter is always attained.

The following theorem gives us a characterization of finite homogeneous ultrametric spaces in terms of representing trees, see Figure~\ref{fig3}.

\begin{theorem}\label{t23}
Let $X$ be a finite ultrametric space with the representing tree $(T_X,l_X)$. The space $X$ is homogeneous if and only if the following two conditions hold.
\begin{itemize}
  \item [\textup{(i)}] For all different $x,y \in V(T_X)$ with $\operatorname{lev}(x)=\operatorname{lev}(y)$ the equality $l_X(x)=l_X(y)$ holds.
  \item [\textup{(ii)}] For all different $x,y \in V(T_X)$ with $\operatorname{lev}(x)=\operatorname{lev}(y)$ the equality $\delta^{+}(x)=\delta^{+}(y)$ holds.
\end{itemize}
\end{theorem}

\begin{proof}
Let $X$ be homogenous.  It is clear that for any $x\in X$ the set $\operatorname{Spec}(X,x)$ coincide with the set of all labels on the unique path from the root of $T_X$ to the leaf $\{x\}$. Hence, according to condition (1) of Theorem~\ref{t31} and to the construction of $T_X$ (all the labels on a path from the root of $T_X$ to any leaf strictly decrease) we immediately obtain condition (i) 
and the fact that all the leaves of $T_X$ are on one and the same level.
Further, let $B_1$ and $B_2$ be any two balls with equal diameters. According to Proposition~\ref{lbpm}  there exist two inner nodes $b_1$ and $b_2$ of the tree $T_X$ such that $L_{T_{b_1}}=B_1$ and $L_{T_{b_2}}=B_2$. According to condition (i) the nodes $b_1$ and $b_2$ are on one and the same level. By condition (2) of Theorem~\ref{t31} $B_1$ and $B_2$ are isometric. Hence, by Theorem~\ref{l3.3} the subtrees $T_{b_1}$ and $T_{b_2}$ are isomorphic and this means that $\delta^{+}(b_1)=\delta^{+}(b_2)$ which establishes condition (ii).

The converse implication easily follows from the construction of representing trees and from Proposition~\ref{lbpm}.
\end{proof}

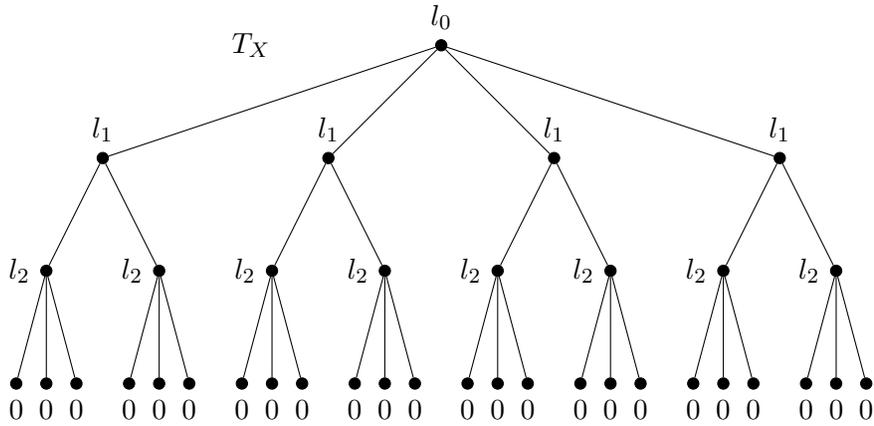
\begin{figure}[h]
\begin{center}
\begin{tikzpicture}
\tikzstyle{level 1}=[level distance=15mm,sibling distance=30mm]
\tikzstyle{level 2}=[level distance=15mm,sibling distance=15mm]
\tikzstyle{level 3}=[level distance=15mm,sibling distance=4mm]
\tikzset{
solid node/.style={circle,draw,inner sep=1.5,fill=black},
hollow node/.style={circle,draw,inner sep=1.5}
}

\node [label=left:{\(T_X\)}] at (-2,0) {};
\node (1) [solid node, label=above:{\(l_0\)}] at (0,0) {}
child {node[solid node, label=above:{\(l_1\)}]{}
        child{node[solid node, label=left:{\(l_{2}\)}] {}
            child{node[solid node, label=below:{\(0\)}] {} }
            child{node[solid node, label=below:{\(0\)}] {} }
            child{node[solid node, label=below:{\(0\)}] {} }
        }
    	child{node[solid node, label=left:{\(l_{2}\)}] {}
            child{node[solid node, label=below:{\(0\)}] {} }
            child{node[solid node, label=below:{\(0\)}] {} }
            child{node[solid node, label=below:{\(0\)}] {} }
        }
      }
child {node[solid node, label=above:{\(l_1\)}]{}
        child{node[solid node, label=left:{\(l_{2}\)}] {}
            child{node[solid node, label=below:{\(0\)}] {} }
            child{node[solid node, label=below:{\(0\)}] {} }
            child{node[solid node, label=below:{\(0\)}] {} }
        }
    	child{node[solid node, label=left:{\(l_{2}\)}] {}
            child{node[solid node, label=below:{\(0\)}] {} }
            child{node[solid node, label=below:{\(0\)}] {} }
            child{node[solid node, label=below:{\(0\)}] {} }
        }
      }
child {node[solid node, label=above:{\(l_1\)}]{}
        child{node[solid node, label=left:{\(l_{2}\)}] {}
            child{node[solid node, label=below:{\(0\)}] {} }
            child{node[solid node, label=below:{\(0\)}] {} }
            child{node[solid node, label=below:{\(0\)}] {} }
        }
    	child{node[solid node, label=left:{\(l_{2}\)}] {}
            child{node[solid node, label=below:{\(0\)}] {} }
            child{node[solid node, label=below:{\(0\)}] {} }
            child{node[solid node, label=below:{\(0\)}] {} }
        }
      }
child {node[solid node, label=above:{\(l_1\)}]{}
        child{node[solid node, label=left:{\(l_{2}\)}] {}
            child{node[solid node, label=below:{\(0\)}] {} }
            child{node[solid node, label=below:{\(0\)}] {} }
            child{node[solid node, label=below:{\(0\)}] {} }
        }
    	child{node[solid node, label=left:{\(l_{2}\)}] {}
            child{node[solid node, label=below:{\(0\)}] {} }
            child{node[solid node, label=below:{\(0\)}] {} }
            child{node[solid node, label=below:{\(0\)}] {} }
        }
      };

\end{tikzpicture}
\end{center}
\caption{An example of the labeled representing tree $T_X$ of a finite homogeneous ultrametric space.}
\label{fig3}
\end{figure}

Analyzing the structural properties of representing trees of homogeneous finite ultrametric spaces it is possible also to distinguish the following two classes of finite ultrametric spaces: the spaces $X$ for which all the leaves of $T_X$ are on the same level and the spaces $X$ for which all the labels of inner nodes of $T_X$ being at the same level are equal.

The proofs of the following two propositions follow directly from the fact that for every $x\in X$ the set $\operatorname{Spec}(X,x)$ coincides with the set of labels of vertices lying at the unique path from the leaf $\{x\}$ to the root of $T_X$ and the fact that these labels monotonically increase along this path.

\begin{proposition}\label{p210}
Let $(X,d)$ be a finite ultrametric space with $|X|\geqslant 2$. The following conditions are equivalent.
\begin{itemize}
  \item [\textup{(i)}] All leaves of $T_X$ are on the same level.
  \item [\textup{(ii)}] $|\operatorname{Spec}(X,x)|= |\operatorname{Spec}(X,x')|$ for all $x, x' \in X$.
\end{itemize}
\end{proposition}

\begin{proposition}\label{p211}
Let $(X,d)$ be a finite ultrametric space with $|X|\geqslant 2$ and let $\Sp{X}=\{0,s_1,...,s_n\}$ with $s_1<\cdots <s_n$.
The following conditions are equivalent.
\begin{itemize}
  \item [\textup{(i)}] All labels of the inner nodes of $T_X$ being at the same level are equal.
  \item [\textup{(ii)}] For every $x\in X$ we have $\operatorname{Spec}(X,x)= \{0,s_k,s_{k+1},...,s_n\}$ for some $k\in \{1,...,n\}$.
\end{itemize}
\end{proposition}

The previous two propositions immediately give the following.

\begin{corollary}\label{c211}
Let $(X,d)$ be a finite ultrametric space with $|X|\geqslant 2$.
The following conditions are equivalent.
\begin{itemize}
  \item [\textup{(i)}] All labels of the inner nodes of $T_X$ being at the same level are equal and all leaves of $T_X$ are on the same level.
  \item [\textup{(ii)}]  $\operatorname{Spec}(X,x)= \operatorname{Spec}(X,x')$ for all $x, x' \in X$.
\end{itemize}
\end{corollary}

In the following proposition we consider a slight modification of representing trees of finite homogeneous ultrametric spaces.
First, we preserve for such trees condition (i) of Theorem~\ref{t23} and instead of condition (ii) we consider that all leaves of a representing tree are at the same level.

\begin{proposition}\label{t210}
Let $(X,d)$ be a finite ultrametric space with $|X|\geqslant 2$ and let all leaves of $T_X$ be on the same level. The following conditions are equivalent.
\begin{itemize}
  \item [\textup{(i)}] All labels of the inner nodes of $\ T_X$ being at the same level are equal.
  \item [\textup{(ii)}] $V(G'_{r,X})=X$ for every nonzero $r\in \Sp{X}$.
\end{itemize}
\end{proposition}
\begin{proof}
(i)$\Rightarrow$(ii) Let all inner nodes $x_1,...,x_p$ being at the same level are labeled by $r$.
We can easily see that $L_{T_{x_1}},...,L_{T_{x_p}}=X$. Using this fact and Lemma~\ref{c25} we obtain a desired implication.

(ii)$\Rightarrow$(i) In the case where the number of levels of $T_X$ is equal to 1 this implication is evident. Suppose the number of levels is more than 1.  Let condition (ii) hold and let $x$ and $y$ be some inner nodes being at the first level such that $l(x)\neq l(y)$. Without loss of generality suppose $l(x)>l(y)$. If $V(G'_{l(y),X})\neq X$ then we have a contradiction. Assume that $V(G'_{l(y),X})= X$. Consider the graph $G'_{l(x),X}$. Since the nodes $x$ and $y$ are on the same level it is easy to see that $L_{T_y}\not\subseteq V(G'_{l(x),X})$ which contradicts to (ii). Hence $l(x)=l(y)$. Arguing as above for the nodes of next levels we complete the proof.
\end{proof}
\begin{corollary}
Let $(X,d)$ be a finite homogeneous ultrametric space. Then $V(G'_{r,X})=X$ for every $r\in \Sp{X}\setminus \{0\}$.
\end{corollary}

\section{Spaces for which the representing trees are perfect}

Following~\cite{DADS} we shall call strictly $n$-ary tree $T$ \emph{perfect} if all leaves of $T$ are on the same level.
It is clear that in general case the representing trees of finite homogeneous ultrametric spaces are not perfect.
They are perfect if condition (ii) of Theorem~\ref{t23} is replaced by the following more strict condition: for every inner node $x$ of the representing tree $T_X$ we have $\delta^+(x)=n$.

The aim of this section is to describe spaces for which the representing trees are perfect in terms of graphs $G'_{r,X}$. In the following proposition we have a description of such spaces in the special case when the internal labeling of $T_X$ is injective.

\begin{figure}[ht]
\begin{center}
\begin{tikzpicture}
\tikzstyle{level 1}=[level distance=10mm,sibling distance=30mm]
\tikzstyle{level 2}=[level distance=10mm,sibling distance=10mm]
\tikzstyle{level 3}=[level distance=15mm,sibling distance=3mm]
\tikzset{
solid node/.style={circle,draw,inner sep=1.5,fill=black},
hollow node/.style={circle,draw,inner sep=1.5}
}

\node [label=left:{\(T\)}] at (-2,0) {};
\node (1) [solid node, label=above:{}] at (0,0) {}
child {node[solid node, label=left:{}]{}
        child{node[solid node, label=left:{}] {}
            child{node[solid node, label=below:{}] {} }
            child{node[solid node, label=below:{}] {} }
            child{node[solid node, label=below:{}] {} }
        }
    	child{node[solid node, label=left:{}] {}
            child{node[solid node, label=below:{}] {} }
            child{node[solid node, label=below:{}] {} }
            child{node[solid node, label=below:{}] {} }
        }
        child{node[solid node, label=left:{}] {}
            child{node[solid node, label=below:{}] {} }
            child{node[solid node, label=below:{}] {} }
            child{node[solid node, label=below:{}] {} }
        }
      }
child {node[solid node, label=left:{}]{}
        child{node[solid node, label=left:{}] {}
            child{node[solid node, label=below:{}] {} }
            child{node[solid node, label=below:{}] {} }
            child{node[solid node, label=below:{}] {} }
        }
    	child{node[solid node, label=left:{}] {}
            child{node[solid node, label=below:{}] {} }
            child{node[solid node, label=below:{}] {} }
            child{node[solid node, label=below:{}] {} }
        }
        child{node[solid node, label=left:{}] {}
            child{node[solid node, label=below:{}] {} }
            child{node[solid node, label=below:{}] {} }
            child{node[solid node, label=below:{}] {} }
        }
      }
child {node[solid node, label=left:{}]{}
        child{node[solid node, label=left:{}] {}
            child{node[solid node, label=below:{}] {} }
            child{node[solid node, label=below:{}] {} }
            child{node[solid node, label=below:{}] {} }
        }
        child{node[solid node, label=left:{}] {}
            child{node[solid node, label=below:{}] {} }
            child{node[solid node, label=below:{}] {} }
            child{node[solid node, label=below:{}] {} }
        }
    	child{node[solid node, label=left:{}] {}
            child{node[solid node, label=below:{}] {} }
            child{node[solid node, label=below:{}] {} }
            child{node[solid node, label=below:{}] {} }
        }
      };
\end{tikzpicture}
\caption{An example of a perfect strictly 3-ary tree $T$.}
\label{fig4}
\end{center}
\end{figure}
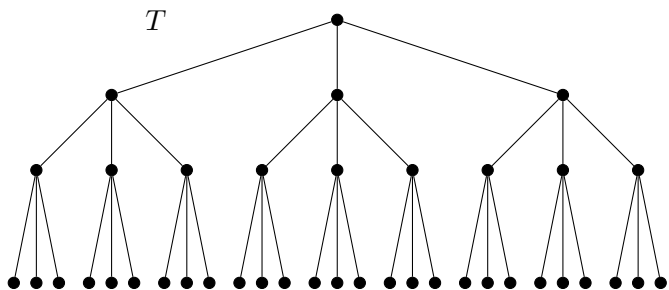

\begin{proposition}\label{t4}
Let $(X,d)$ be a finite ultrametric space with $|X|\geqslant 2$ and let $T_X$ be its representing tree such that all the labels of different internal nodes of $T_X$ are different. The following conditions are equivalent.
\begin{itemize}
\item [\textup{(i)}] $T_X$ is a perfect strictly $n$-ary tree.
\item [\textup{(ii)}] $G_{r,X}'= G_{r,X}'[X_1,...,X_n]$ with $|X_1|=\cdots =|X_n|$ for every nonzero $r\in \Sp{X}$.
\end{itemize}
\end{proposition}
\begin{proof}
(i)$\Rightarrow$(ii) Let $T_X$ be a perfect strictly $n$-ary tree such that all labels of $T_X$ are  different.  According to Lemma~\ref{c25} we have $G'_{r,X}=G'_{r,X}[L_{T_{x_1}},....L_{T_{x_n}}]$ for every $r\in \Sp{X}\setminus \{0\}$, where $x_1,...,x_n$ are the direct successors of the node $x$ labeled by $r$. Since $T_X$ is perfect it is easy to see that the trees  $T_{x_1},...,T_{x_n}$ are isomorphic as rooted trees. Hence $|L_{T_{x_1}}|=\cdots=|L_{T_{x_n}}|$ which is equivalent to condition (ii).

(ii)$\Rightarrow$(i) Let us prove this implication by induction on the number of levels of the tree $T_X$. Let the number of levels of $T_X$ be equal to 1. It is evident that in this case implication (ii)$\Rightarrow$(i) holds.

Suppose that the implication (ii)$\Rightarrow$(i) holds in the case when the number of levels of $T_X$ is equal to $k$ and suppose that the number of levels of $T_X$ is equal to $k+1$ and condition (ii) holds. Let $x$ be a root of $T_X$ and $x_1,...,x_n$ be its direct successors.  Consider the subtrees $T_{x_1},...,T_{x_n}$ with the roots $x_1,...,x_n$ and let $X_1=L_{T_{x_1}},...,X_n=L_{T_{x_n}}$. Since all the labels of $T_X$ are different and condition (ii) holds for every $r\in \Sp{X}\setminus\{0\}$ it follows that condition (ii) also holds for the subspaces $(X_1,d),...,(X_n,d)$. According to the supposition of induction all the trees $T_{X_1},...,T_{X_n}$ are perfect strictly $n$-ary since the number of levels of the trees $T_{x_1},...,T_{x_n}$ is equal to $k$. Taking into consideration the construction of $T_X$ from the trees $T_{x_1},...,T_{x_n}$ it easy to see that $T_X$ is also a perfect strictly $n$-ary tree.
\end{proof}

Omitting the condition ``different internal nodes of $T_X$ are different'' we can generalize Theorem~\ref{t4} to the following.
\begin{theorem}\label{t5}
Let $(X,d)$ be a finite ultrametric space with $|X|\geqslant 2$. The following conditions are equivalent.
\begin{itemize}
\item [\textup{(i)}] $T_X$ is a perfect strictly $n$-ary tree.
\item [\textup{(ii)}] $G_{r,X}'$ is a union of a finite number of complete multipartite graphs having the same number of vertices in each part for every nonzero $r\in \Sp{X}$.
\end{itemize}
\end{theorem}
\begin{proof}
Implication (i)$\Rightarrow$(ii) easily follows form Lemma~\ref{c25}.

(ii)$\Rightarrow$(i) Let condition (ii) hold, $(T_X,l_X)$ be a representing tree of the space $(X,d)$ and let $(T_X,\tilde{l}_X)$ be the same tree with another labeling function $\tilde{l}_X$ having the following properties: labels of different internal nodes are different and labels monotonically decrease along all paths from the root of $T_X$ to any leaf. Taking into consideration Lemma~\ref{c25} it is easy to see that condition (ii) of Proposition~\ref{t4} holds for the new ultrametric space $\tilde{X}$ with the representing tree $(T_X,\tilde{l}_X)$. Since the structure of $T_X$ was not changed while changing the labeling function by Proposition~\ref{t4} we have that $T_X$ is a perfect strictly $n$-ary tree.
\end{proof}

\section{Ultrametric spaces generated by unrooted labeled trees}\label{sec4}
Throughout this section by \(T = T(l)\) we denote an unrooted labeled tree with the labeling function $l\colon V(T)\to \RR^+$.
In the paper~\cite{Do20} it was defined a mapping \(d_l \colon V(T) \times V(T) \to \RR^{+}\) as
\begin{equation}\label{e11.3}
d_l(u, v) = \begin{cases}
0 & \text{if } u = v,\\
\max\limits_{v^{*} \in V(P)} l(v^{*}) & \text{if } u \neq v,
\end{cases}
\end{equation}
where \(P\) is a path joining \(u\) and \(v\) in \(T(l)\). Also the following proposition was shown there.

\begin{proposition}[\cite{Do20}]\label{p11.9}
The following statements are equivalent for every labeled tree \(T = T(l)\).
\begin{itemize}
  \item [\textup{(i)}] The function $d_l$ is an ultrametric on $V(T)$.
  \item [\textup{(ii)}] The inequality
\begin{equation}\label{t11e1}
\max\{l(u_1), l(v_1)\} > 0
\end{equation}
holds for every \(\{u_1, v_1\} \in E(T)\).
\end{itemize}
 \end{proposition}
It is possible to show that in the case when condition~(\ref{t11e1}) does not hold the space $X$ is pseudoultrametric, see~\cite{DK21} for the details.

In this section we show that not every finite ultrametric space $(X,d)$ can be represented by a labeled tree $T(l)$.  Thus, labeled unrooted tress $T(l)$ with labeling functions $l$ satisfying~(\ref{t11e1}) generate a new class of finite ultrametric spaces.

We need the following lemma for the proof of the main result of this section.
\begin{lemma}\label{l1.3}
Let $X$, $|X|\geqslant 2$, be a finite ultrametric space and let $u$ be an inner node of $T_X$. Then $u$ has at least one direct successor which is a leaf if and only if in the ball $B=L_{T_u}$ there exists a point $z$ such that $d(z,t)=\diam B$ for all $t \in B\setminus \{z\}$.
\end{lemma}
\begin{proof}
The necessity follows directly from Theorem~\ref{t2.9} and Proposition~\ref{lbpm}. As the point $z$ one can chose any leaf of the tree $T_X$ which is a direct successor of the inner node $u$. The sufficiency can be easily shown by contradiction.
\end{proof}

\begin{theorem}\label{t1.4}
Let \(T = T(l)\), $|V(T)|\geqslant 2$, be a labeled tree such that \(X=(V(T), d_l)\) is an ultrametric space. Then in the representing tree $T_X$ every inner node has at least one direct successor which is a leaf.

Conversely, let $X$, $|X|\geqslant 2$, be an ultrametric space such that in the representing tree $T_X$ every inner node has at least one direct successor which is a leaf. Then $X$ can be represented by a labeled unrooted tree \(T = T(l)\).
\end{theorem}
\begin{proof}
Let $B=B(x,r)$ be any ball from $X$ with the center $x$ and the diameter $r$. Since $X$ is finite, without loss of generality, we consider that diameter is attained, i.e., there exists $y\in B$  such that $d_l(x,y)=r$.
According to~(\ref{e11.3}) there exists a vertex $z\in V(P)$ such that $l(z)=r$ (possibly $z=x$ or $z=y$), where $P$ is the unique path connecting $x$ and $y$ in $T(l)$. Since $d_l(x,z)=r$ we have $z\in B$.

Let us describe the set $B$ on the tree $T(l)$. According to the definition of $d_l$ we have $t\in B$, $t\neq x$,  if
$
\max\limits_{v^*\in V(P)}l(v^*)\leqslant r
$
where $P$ is a path joining $x$ and $t$ in $T(l)$.
Hence, using~(\ref{e11.3}) and the fact that $r=l(z)$ is the maximal label among all labels of vertices of $T$ belonging to $B$, we see that $d_l(t,z)=r$ for all $t\in B\setminus \{z\}$. Thus, it follows from Lemma~\ref{l1.3} and Proposition~\ref{lbpm} that every inner note of $T_X$ has at least one direct successor which is a leaf.

Conversely, let us now describe a procedure of construction of the labeled tree $T(l)$ from the representing tree $(T_X,l_X)$ in which every inner node satisfies the above mentioned property. Let the root $r$ of $T_X$ have the label $l_X(r)$ and let $r_1,...,r_n$ be direct successors of $r$ which are leaves, $n\geqslant 1$,  and  $s_1,...,s_k$ be direct successors of $r$ which are inner nodes, $k\geqslant 1$, with the respective labels $l_X(s_1),...,l_X(s_k)$. Set $V(T(l)):=\{r_1,...,r_n\}$,
$E(T(l)):=\{  \{r_1,r_2\},...,\{r_{n-1},r_n\} \}$,
$l(r_1)=\cdots = l(r_n):=l_X(r)$.

Further, let $t_{11},...,t_{1n_1}$ be the direct successors of $s_1$ which are leaves, $n_1\geqslant 1$,  and  $u_{11},...,u_{1p_1}$ be direct successors of $s_1$ which are inner nodes, $k_1\geqslant 0$,\ldots,
$t_{k1},...,t_{kn_k}$ be the direct successors of $s_k$ which are leaves, $n_k\geqslant 1$,  and  $u_{k1},...,u_{1p_k}$ be the direct successors of $s_k$ which are inner nodes, $p_k\geqslant 0$.

Set $$V(T(l)):=V(T(l))
\cup\{t_{11},...,t_{1n_{1}}\}
\cup...
\cup\{t_{k1},...,t_{kn_{k}}\},$$

$$E(T(l)):=E(T(l))\cup\{ \{r_n, t_{11}\},\{t_{11},t_{12}\},...,\{t_{1n_1-1},t_{1n_1}\}\} \},$$
$$
...
$$
$$E(T(l)):=E(T(l))\cup\{ \{r_n, t_{k1}\},\{t_{k1},t_{k2}\},...,\{t_{kn_k-1},t_{kn_k}\}\} \},$$
and
$l(t_{11})=...=l(t_{1n_{1}}):=l_X(s_1)$,
...
$l(t_{k1})=...=l(t_{kn_{k}}):=l_X(s_k)$.
Thus, the degree of $r_n$ is equal to $1+k$. After that we look at the successors of the inner nodes $u_{11},...,u_{1{p_1}}$,...,$u_{k1},...,u_{1p_k}$ and repeat the same procedure. Examples of a representing tree $(T_X,l_X)$ and the corresponding tree $T(l)$ are depicted in Figure~\ref{fig11} and in Figure~\ref{fig13}, respectively. In order to finish the proof it suffices only to note that the ultrametric $d_l$ defined by~(\ref{e11.3}) indeed coincides with the ultrametric $d$.
\end{proof}

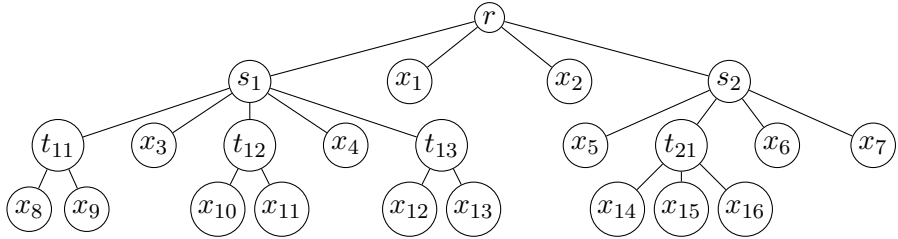
\begin{figure}[ht]
\begin{center}
\begin{tikzpicture}[
node distance=1cm, on grid=true,
level 1/.style={level distance=1cm,sibling distance=2.5cm},
level 2/.style={level distance=1cm,sibling distance=1.5cm},
level 3/.style={level distance=1cm,sibling distance=.9cm},
solid node/.style={circle,draw,inner sep=1.5,fill=black},
hollow node/.style={circle,draw,inner sep=1.5},
scale = 0.85
]

\node [hollow node, label={[label distance=20pt]
left:{}}] 
(A2) at (3,0) {\(r\)}
	child {node [hollow node]{\(s_1\)}
		child {node [hollow node]{\(t_{11}\)}
			child {node [hollow node]{\(x_{8}\)}}
			child {node [hollow node]{\(x_{9}\)}}
		}
    	child {node [hollow node]{\(x_3\)}}
		child  {node [hollow node]{\(t_{12}\)}
			child [sibling distance=1cm] {node [hollow node]{\(x_{10}\)}}
			child [sibling distance=1cm] {node [hollow node]{\(x_{11}\)}}
		}
	    child {node [hollow node]{\(x_4\)}}
		child  {node [hollow node]{\(t_{13}\)}
			child [sibling distance=1cm] {node [hollow node]{\(x_{12}\)}}
			child [sibling distance=1cm] {node [hollow node]{\(x_{13}\)}}
		}
	}
	child {node [hollow node]{\(x_1\)}}
	child {node [hollow node]{\(x_2\)}}
	child {node [hollow node]{\(s_2\)}
    	child {node [hollow node]{\(x_5\)}}
		child {node [hollow node]{\(t_{21}\)}
			child [sibling distance=1cm] {node [hollow node]{\(x_{14}\)}}
			child [sibling distance=1cm] {node [hollow node]{\(x_{15}\)}}
			child [sibling distance=1cm] {node [hollow node]{\(x_{16}\)}}
		}
        child {node [hollow node]{\(x_6\)}}
    	child {node [hollow node]{\(x_7\)}}
	};
\end{tikzpicture}
\caption{The canonical representing tree $(T_X,l_X)$ of a finite ultrametric space $(X,d)$ with $X=\{x_1,...,x_{16}\}$.}
\label{fig11}
\end{center}

\end{figure}

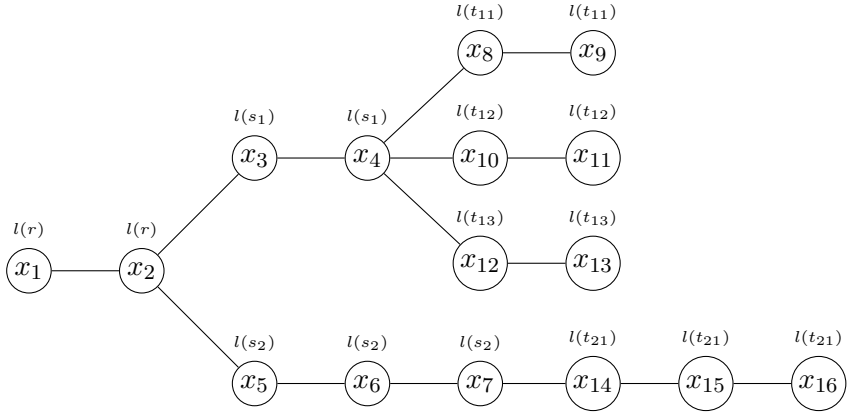
\begin{figure}[ht]
\begin{center}
\begin{tikzpicture}[
node distance=1.5cm, on grid,
hollow node/.style={circle,draw,inner sep=1.5}
]
\tikzset{solid/.style={circle,draw,inner sep=1.5,fill=black}}

\node [hollow node, label=above:\tiny{$l(r)$}] (x1) {$x_1$};
\node [hollow node, label=above:\tiny{$l(r)$}] (x2) [right=of x1] {$x_2$};

\node [hollow node, label=above:\tiny{$l(s_1)$}] (x3) [right=of x2, yshift=1.5cm] {$x_3$};
\node [hollow node, label=above:\tiny{$l(s_1)$}] (x4) [right=of x3] {$x_4$};

\node [hollow node, label=above:\tiny{$l(t_{11})$}] (x8) [right=of x4, yshift=1.4cm] {$x_8$};
\node [hollow node, label=above:\tiny{$l(t_{11})$}] (x9) [right=of x8] {$x_{9}$};
\node [hollow node, label=above:\tiny{$l(t_{12})$}] (x10) [right=of x4] {$x_{10}$};
\node [hollow node, label=above:\tiny{$l(t_{12})$}] (x11) [right=of x10] {$x_{11}$};
\node [hollow node, label=above:\tiny{$l(t_{13})$}] (x12) [right=of x4, yshift=-1.4cm] {$x_{12}$};
\node [hollow node, label=above:\tiny{$l(t_{13})$}] (x13) [right=of x12] {$x_{13}$};

\draw (x4)--(x8);
\draw (x4)--(x10);
\draw (x4)--(x12);
\draw (x8)--(x9);
\draw (x10)--(x11);
\draw (x12)--(x13);


\node [hollow node, label=above:\tiny{$l(s_2)$}] (x5) [right=of x2, yshift=-1.5cm] {$x_5$};
\node [hollow node, label=above:\tiny{$l(s_2)$}] (x6) [right=of x5] {$x_6$};
\node [hollow node, label=above:\tiny{$l(s_2)$}] (x7) [right=of x6] {$x_7$};

\node [hollow node, label=above:\tiny{$l(t_{21})$}] (x14) [right=of x7] {$x_{14}$};
\node [hollow node, label=above:\tiny{$l(t_{21})$}] (x15) [right=of x14] {$x_{15}$};
\node [hollow node, label=above:\tiny{$l(t_{21})$}] (x16) [right=of x15] {$x_{16}$};

\draw (x1)--(x2)--(x3)--(x4);
\draw (x2)--(x5)--(x6)--(x7);
\draw (x7)--(x14)--(x15)--(x16);

\end{tikzpicture}
\end{center}
\caption{The labeled tree $T(l)$ of the space $(X,d)$.}
\label{fig13}
\end{figure}

\section{Some classes of finite ultrametric spaces}\label{s5}
In this section we give a detailed survey of some special classes of finite ultrametric spaces, which were considered in the last ten years.
The structure of representing trees will be mainly used in studying of hereditary properties of spaces from these classes.
The images of representing trees of the spaces discussed in this section can be found in the corresponding references.

\noindent\textbf{2.1. Spaces extremal for the Gomory-Hu inequality.}
In 1961 E.\,C.~Gomory and T.\,C.~Hu \cite{GomoryHu(1961)} for arbitrary finite ultrametric space $X$ proved the inequality \mbox{$|\operatorname{Sp}(X)| \leqslant |X|$}. Define by $\mfu$ the class of finite ultrametric spaces $X$ such that $\abs{\Sp{X}} = \abs{X}$. In~\cite{PD(UMB)} two descriptions of $X \in \mfu$ were obtained in terms of graphs $G'_{r,X}$ (see Definitions~\ref{d14}, ~\ref{d15}) and in terms of representing trees (see Theorem~\ref{t22} below).
In~\cite{DP19} it was proved that  \(X\in \mfu\) if and only if there are no equilateral triangles in \(X\) and the graph \(G'_{r, X}\) is connected for every nonzero \(r \in \Sp{X}\). Another one criterium of $X \in \mfu$ in terms of weighted Hamiltonian cycles and weighted Hamiltonian paths was proved in~\cite{DPT}.

\begin{theorem}[\!\cite{PD(UMB)}]\label{t22}
Let $(X, d)$ be a finite ultrametric space with $|X| \geqslant 2$. The following condidtions are equivalent.
\begin{enumerate}
  \item [\textup{(i)}] $(X, d) \in \mathfrak U$.
  \item [\textup{(ii)}] $G'_{r,X}$ is complete bipartite for every  nonzero $r\in \operatorname{Sp}(X)$.
  \item [\textup{(iii)}] $T_X$ is strictly binary and the labels of different internal nodes are different.
\end{enumerate}
\end{theorem}

The following corollary of Theorem~\ref{t22} in fact states that subspaces of the space from the class $\mfu$ inherit the class.
\begin{corollary}[\!\!\cite{PD(UMB)}]\label{c22}
Let $X \in \mfu$ and $Y$ be a nonempty subspace of $X$. Then $Y \in \mfu$.
\end{corollary}

\noindent\textbf{2.2. Spaces for which the labels of different internal nodes of the representing trees are different.}
Omitting in statement (iii) of Theorem~\ref{t22} the condition ``$T_X$ is strictly binary'' we obtain the class of finite ultrametric spaces $X$ for which the labels of different internal nodes of $T_X$ are different.
In the following theorem we combine results obtained in~\cite{DP20} and~\cite{DP19}.
\begin{theorem}[\!\!\cite{DP20, DP19}]\label{t1}
Let $(X,d)$ be a finite ultrametric space with $|X|\geqslant 2$. The following conditions are equivalent.
\begin{itemize}
\item[\textup{(i)}] The labels of different internal nodes of $T_X$ are different.
\item[\textup{(ii)}] The graph $G'_{r,X}$ is complete multipartite for every nonzero $r\in \Sp{X}$.
\item[\textup{(iii)}] The graph \(G'_{r,X}\) is connected for every nonzero \(r \in \Sp{X}\).
\item[\textup{(iv)}]  The diameters of different nonsingular balls are different.
\item[\textup{(v)}]  The equality
$$
|\Sp{X}| = |\mathbf{B}_X| - |X| + 1
$$
holds.
\end{itemize}
\end{theorem}

\noindent\textbf{2.3. Spaces for which the representing trees are strictly binary.}
Omitting in statement (iii) of Theorem~\ref{t22} the condition ``the labels of different internal nodes are different'' we obtain the class of finite ultrametric spaces having the strictly binary representing trees.

In the following we identify a finite ultrametric space $(X, d)$ with a complete weighted graph $G_X = (G_X, w)$, \(w \colon E(G_X) \to \mathbb{R}^{+}\), such that $V(G_X) = X$ and $w(\{x,y\}) = d(x,y)$
for all different \(x\), \(y \in X\).

\begin{proposition}[\!\!\cite{DPT}]\label{p10}
Let $(X,d)$ be a finite nonempty ultrametric space. The following conditions are equivalent.
\begin{itemize}
\item[\textup{(i)}] $T_X$ is strictly binary.
\item[\textup{(ii)}] If $Y\subseteq X$ and $|Y|\geqslant 3$, then there exists a Hamilton cycle $C \subseteq G_{Y}$ with exactly two edges of maximal weight.
\item[\textup{(iii)}] There are no equilateral triangles in $(X,d)$.
\end{itemize}
\end{proposition}

\noindent\textbf{2.4. Spaces for which the representing trees are strictly $n$-ary.}
Let $(X,d)$ be a metric space. Recall that balls $B_1$, $\ldots$, $B_k$ in $(X,d)$ are \emph{equidistant} if there is $r>0$ such that $d(x_i, x_j)=r$ holds whenever $x_i \in B_i$ and $x_j \in B_j$ and $1\leqslant i< j \leqslant k$. Every two disjoint balls in any ultrametric space are equidistant.

\begin{theorem}[\!\cite{DP20}]\label{t27}
Let $(X,d)$ be a finite ultrametric space with $|X| \geqslant 2$ and let $n\geqslant 2$ be integer. The following conditions are equivalent.
\begin{itemize}
\item [\textup{(i)}] $T_X$ is strictly $n$-ary.
\item [\textup{(ii)}] For every nonzero $t\in \Sp{X}$, the graph $G'_{t,X}$ is the union of $p$ complete $n$-partite graphs, where $p$ is the number of internal nodes of $T_X$ labeled by~$t$.
\item [\textup{(iii)}] For every nonsingular ball $B \in \mathbf B_X$, there are equidistant disjoint balls $B_1,...,B_n \in \mathbf B_X$ such that $B=\bigcup\limits_{j=1}^n B_j$.
\item [\textup{(iv)}] The equality
\begin{equation*}
(n-1)|\mathbf{B}_Y| +1 = n |Y|
\end{equation*}
holds for every ball $Y\in \mathbf B_X$.
\end{itemize}
\end{theorem}

Write
$$
\Delta^+(T) := \max_{v \in V(T)} \delta^+(v).
$$
\begin{corollary}[\!~\cite{DP20}]\label{c2.10}
The inequality
\begin{equation*}
|\mathbf{B}_X| \geqslant \frac{\Delta^+(T_X)|X|-1}{\Delta^+(T_X)-1}
\end{equation*}
holds for every finite nonempty ultrametric space $(X,d)$. This inequality becomes an equality if and only if $T_X$ is a strictly $n$-ary rooted tree with $n = \Delta^+(T_X)$.
\end{corollary}

\begin{proposition}[\!\!\cite{DP20}]\label{p2.11}
Let $(X, d)$ be a finite ultrametric space with $|X|\geq 2$. Then the inequality
\begin{equation*}
2|\mathbf{B}_X| \geqslant |\operatorname{Sp}(X)| + \frac{2\Delta^+(T_X)|X|-\Delta^+(T_X)-|X|}{\Delta^+(T_X) - 1}
\end{equation*}
holds. This inequality becomes an equality if and only if $T_X$ is a strictly $n$-ary rooted tree with injective internal labeling and $n=\Delta^+(T_X)$.
\end{proposition}

\noindent\textbf{2.5. Ultrametric spaces which are as rigid as possible.} Let $(X, d)$ be a metric space and let $\operatorname{Iso}(X)$ be the group of isometries of $(X, d)$. We say that $(X, d)$ is \emph{rigid}  if $|\operatorname{Iso}(X)|=1$. It is clear that $(X, d)$ is rigid if and only if $g(x)=x$ for every $x \in X$ and every $g \in \Iso(X)$.

For every self-map $f\colon X\to X$ we denote by $\Fix(f)$ the set of fixed points of $f$. Using this denotation we obtain that a finite metric space $(X, d)$ is rigid if and only if
$$
\min_{g \in \operatorname{Iso}(X)} |\Fix(g)|=|X|.
$$
 It is easy to show that the finite ultrametric spaces $X$ with $|X| \geq 2$ are not rigid since for every such $X$ there is a self-isometry having exactly $|X|-2$ fixed points, see Proposition 3.2 in~\cite{DPT(Howrigid)}.

If a metric space $(X, d)$ is finite, nonempty and nonrigid, then the inequality
\begin{equation*}
\min_{g \in \operatorname{Iso}(X)} |\Fix(g)| \leq |X|-2
\end{equation*}
holds, because the existence of $|X|-1$ fixed points for $g \in \operatorname{Iso}(X)$ implies that $g$ is identical.

The quantity $\min_{g \in \operatorname{Iso}(X)} |\Fix(g)|$ can be considered as a measure for ``rigidness'' for finite metric spaces $(X, d)$. Thus the finite ultrametric spaces satisfying the equality
\begin{equation*}
\min_{g \in \operatorname{Iso}(X)} |\Fix(g)| = |X|-2,
\end{equation*}
are as rigid as possible. Let us denote by $\mathfrak{R}$ the class of all finite ultrametric spaces $(X, d)$ which satisfy this equality. The following theorem gives us a characterization of spaces from the class $\mathfrak{R}$ in terms of representing trees.

\begin{theorem}[\!\!\cite{DPT(Howrigid)}]
Let $(X, d)$ be a finite ultrametric space with $|X| \geq 2$. Then the following statements are equivalent.
\begin{itemize}
\item [\textup{(i)}] $(X, d) \in \mathfrak{R}$.
\item [\textup{(ii)}] $|\Iso(X)|=2$.
\item [\textup{(iii)}] $T_X$ is strictly binary with exactly one inner node at each level except the last level.
\end{itemize}
\end{theorem}

\noindent\textbf{2.6. Weak similarity generating spaces.}  Denote by $\tilde{\mathfrak R}$ the class of finite ultrametric spaces $X$ for which $T_X$ has exactly one inner node at each level except the last level. It is clear that $\mathfrak R$ is a proper subclass of $\tilde{\mathfrak R}$.

\begin{definition}\label{d4.5}
Let  $(X,d)$ and $(Y,\rho)$ be metric spaces. A bijective mapping $\Phi\colon X\to Y$ is a \emph{weak similarity} if there exists a strictly increasing bijection $f\colon \Sp{X}\to \Sp{Y}$ such that the equality
\begin{equation*}
f(d(x,y))=\rho(\Phi(x),\Phi(y))
\end{equation*}
holds for all $x$, $y\in X$. The function $f$ is said to be a \emph{scaling function} of $\Phi$. If $\Phi\colon X\to Y$ is a weak similarity, we write $X \we Y$ and say that $X$ and  $Y$  are \emph{weakly similar}. The pair $(f,\Phi)$ is called a \emph{realization} of $X\we Y$.
\end{definition}

In~\cite{DP2} the notion of weak similarity was introduced in a slightly different but equivalent form.

The next theorem gives a description of finite ultrametric spaces for which the isomorphism of representing  trees implies the weak similarity of the spaces.
\begin{theorem}[\!\!\cite{P18}]
Let $X$ be a finite ultrametric space. Then the following statements are equivalent.
\begin{itemize}
  \item [\textup{(i)}] The implication $(\overline{T}_X\simeq \overline{T}_Y) \Rightarrow (X\we Y)$ holds for every finite ultrametric space $Y$.
  \item [\textup{(ii)}] $X\in \tilde{\mathfrak{R}}$.
\end{itemize}
\end{theorem}

The following lemma in fact states that subspaces of the space from the class $\tilde{\mathfrak{R}}$ inherit the class.

\begin{lemma}[\cite{DP18}]\label{l5.11}
Let $(X,d)$ be a finite ultrametric space with $|X|\geqslant 2$, and let $T_X$ have exactly one internal node at each level except the last level. Then for every $Y\subseteq X$, $|Y|\geqslant 2$, the representing tree $T_Y$ of the space $(Y,d)$ also has exactly one internal node at each level except the last level.
\end{lemma}

Denote by $\mathfrak D$ the class of all finite ultrametric spaces $X$ such that the different internal nodes of $T_X$ have the different labels.
It is clear that $\mathfrak R$ and $\tilde{\mathfrak  R}$ are subclasses of $\mathfrak D$. A question arises whether there exist finite ultrametric spaces $X$, $Y\in \mathfrak D$ which do not belong to the class $\tilde{\mathfrak  R}$ and for which the isomorphism of $\overline{T}_X$ and $\overline{T}_Y$ implies $X\we Y$.

Let us define a rooted tree $T$  with $n$ levels by the following two conditions:

(A) There is only one inner node at the level $k$ of $T$ whenever $k<n-1$.

(B) If $u$ and $v$ are different inner nodes at the level $n-1$ then the numbers of offsprings of $u$ and $v$ are equal.

Denote by $\mct$ the class of all finite ultrametric spaces $X$ for which $T_X$ satisfies conditions (A) and (B).

\begin{theorem}[\!\!\cite{P18}]
Let $X \in \mathfrak D$ be a finite ultrametric space. Then the following statements are equivalent.
\begin{itemize}
  \item [\textup{(i)}] The implication $(\overline{T}_X\simeq \overline{T}_Y) \Rightarrow (X\we Y)$ holds for every finite ultrametric space $Y \in \mathfrak D$.
  \item [\textup{(ii)}] $X \in \mct$.
\end{itemize}
\end{theorem}

\noindent\textbf{2.7. Isometry generating spaces.}
Let us denote by $\mathcal{TSI}$ (tree-spectrum isometric) the class of all finite ultrametric spaces~$(X,d)$ which satisfy the following condition: If $(Y,\rho)$ is a finite ultrametric space such that $\overline{T}_X \simeq \overline{T}_Y$ and $\operatorname{Sp}(X) = \operatorname{Sp}(Y)$, then $(X,d)$ and $(Y,\rho)$ are isometric.

Let $T$ be a rooted tree. The \emph{height} of $T$ is the number of edges on the longest path between the root and a leaf of $T$. The height of $T$ will be denoted by $h(T)$. Thus,
\begin{equation*}
h(T)= \max_{v \in L_T} \operatorname{lev}(v).
\end{equation*}

\begin{theorem}[\!\!\cite{DP18}]\label{t3.7}
Let $T$ be a rooted tree with $\delta^+(u)\geqslant 2$ for every internal node $u$. Then the following conditions are equivalent.
\begin{itemize}
\item[\textup{(i)}] The tree $T$ contains exactly one internal node at the levels $1$, $\ldots$, $h(T)-2$ and at most two internal nodes at the level $h(T)-1$. Moreover, if $u$ and $v$ are different internal nodes with
$$
\operatorname{lev}(u)=\operatorname{lev}(v)=h(T)-1,
$$
then $\delta^+(u)=\delta^+(v)$ holds.
\item[\textup{(ii)}] The statement $\overline{T}_X \simeq T$ implies $(X,d) \in \mathcal{TSI}$ for every finite ultrametric space $(X,d)$.
\end{itemize}
\end{theorem}

\begin{theorem}[\!\!\cite{DP18}]\label{t3.5}
Let $(X,d)$ be a finite ultrametric space with $|X|\geqslant 2$. Suppose that the representing tree $T_X$ has an injective internal labeling. Then the following statements are equivalent.
\begin{enumerate}
\item [\textup{(i)}] The space $(X,d)$ belongs to $\mathcal{TSI}$.
\item [\textup{(ii)}] The equality $\operatorname{lev}(v)=\operatorname{lev}(u)$ implies
$$
\operatorname{lev}(v)=\operatorname{lev}(u)=h(T_X)-1 \text{ and } \delta^+(u)=\delta^+(v)
$$
for all distinct internal nodes $u$, $v \in V(T_X)$.
\end{enumerate}
\end{theorem}

\begin{corollary}\label{c5.3}
Let $(X, d)$ be an ultrametric space with $|X| \leqslant 4$. Then $(X, d)$ belongs to $\mathcal{TSI}$.
\end{corollary}

\begin{remark}
The structures of representing trees of spaces from the classes $\mathcal{TSI}$ and $\tilde{\mathfrak{R}}$ coincide.
\end{remark}

\noindent\textbf{2.8. Spaces admitting ball-preserving mappings.} Let $X$ and $Y$ be nonempty metric spaces. A mapping $F\colon X\to Y$ is ball-preserving if the membership relations
\begin{equation*}
F(Z)\in \textbf{B}_Y \quad \text{and}\quad F^{-1}(W)\in \textbf{B}_X
\end{equation*}
hold for all balls $Z\in \textbf{B}_X$ and all balls $W\in \textbf{B}_Y$, where $F(Z)$ is the image of $Z$ under the mapping $F$ and $F^{-1}(W)$ is the preimage of $W$ under this mapping.

For every finite nonempty ultrametric space $X$ denote by $\mathfrak{F}_1^*(X)$ the class of all finite nonempty ultrametric spaces $Y$ for which there are ball-preserving mappings $F \colon X \to Y$. The next our goal is to describe the finite ultrametric spaces \((X, d)\) which admit ball-preserving mapping \(F \colon Y \to Z\) for every nonempty \(Y \subseteq X\) and each nonempty \(Z \subseteq Y\), i.e.,
\begin{equation*}
(Z, d|_{Z \times Z}) \in \mathfrak{F}_1^{*} (Y)
\end{equation*}
holds for all nonempty \(Y \subseteq X\) and \(Z \subseteq Y\). 

\begin{theorem}[\!~\cite{DP19}]\label{th4.21}
Let $(X, d)$ be a finite ultrametric space with $|X| \geqslant 2$. Then the following statements are equivalent.
\begin{itemize}
\item[\textup{(i)}] We have \((Z, d|_{Z \times Z}) \in \mathfrak{F}_1^*(Y)\) for every nonempty $Y \subseteq X$ and every nonempty \(Z \subseteq Y\).
\item[\textup{(ii)}] $T_X$ is strictly binary and one of the following conditions is satisfied.
\begin{itemize}
\item[\textup{($\mathrm{i_1}$)}] For every inner node $v$ of $T_X$ there is a leaf $\{x\}$ of $T_X$ such that $\{x\}$ is a direct successor of $v$.
\item[\textup{($\mathrm{i_2}$)}] $X$ is the unique node of $T_X$ for which both direct successors are inner nodes of $T_X$.
\end{itemize}
\end{itemize}
\end{theorem}

\section{Class preserving subspaces}

The aim of this section is to distinguish classes of ultrametric spaces such that for every space $X$ from the given class every subspace $Y$ of $X$ also belongs to this class.
In this case we shall say that the subspace $Y$ inherits the class.

Let us first consider consecutively the classes of spaces considered in Section~\ref{s5}.

Corollary~\ref{c22} states that every nonempty subspace of the space $X\in \mathfrak U$ also belongs to the class $\mathfrak U$.

Consider the class of ultrametric spaces $X$ for which all labels of different internal nodes of $T_X$ are different. Let $x_1$ be a leaf of $T_X$ and $v$ be its direct predecessor. There are only two possibilities for the node $v$:
\begin{itemize}
  \item [\textup{(i)}] $v$ has only two direct successors $x_1$ and $x_2$,
  \item [\textup{(ii)}] $v$ has more than two direct successors.
\end{itemize}
Consider a transformation of the representing tree $T_X$ to the representing tree $T_Y$, where $Y=X\setminus \{x_1\}$.
In case (i) the removal of the leaf $x_1$ from the tree will entail the removal of the edges $\{v,x_1\}$ and $\{v,x_2\}$. The node $x_2$ replaces $v$ and the label $l(v)$ disappears.
In case (ii) the removal of the leaf $x_1$ entails only the removal of the edge $\{x_1,v\}$.  In each case all labels of different internal nodes remain different.

Let us consider a class of ultrametric spaces $X$ for which $T_X$ is strictly binary. According to condition (iii) of Proposition~\ref{p10} the equivalent condition is that there are no equilateral triangles in the space $X$. It is clear that there are also no any equilateral triangles in every subspace $Y$ of $X$. Thus $Y$ inherits the class.

It is clear that a removal of one leaf from a strictly $n$-ary tree, \mbox{$n\geqslant 3$}, makes it not strictly $n$-ary. Thus, this class does not have the desired property.

According to Lemma~\ref{l5.11} weak similarity generating spaces $\mathfrak{R}$ and $\tilde{\mathfrak{R}}$ have the desired property since in each case  representing trees have exactly one inner node at each level except the last level.

The class $\mct$ from subsection 2.6 as well as the class $\mathcal{TSI}$ from subsection 2.7  does not have the desired property. The reason is that in general case in the representing trees there are two different internal nodes $u$ and $v$  at the level $h(T_X)-1$ and the number of offsprings of $u$ and $v$ are equal. A removal of one of the offsprings violates the structure of the tree.

It is easy to show that spaces admitting ball-preserving mappings have the desired property (Theorem~\ref{th4.21}). Case (\emph{$i_1$}) follows from Lemma~\ref{l5.11} and condition (iii) of Proposition~\ref{p10}. In case (\emph{$i_2$}) it is possible to apply these lemma and proposition to the subtrees $T_u$, $T_v$, where $u$ and $v$ are direct successor of the root $X$ of the tree $T_X$.

It is easy to see that finite homogeneous ultrametric spaces $X$ do not have the desired property. Clearly, a removal of one leaf from the representing tree $T_X$ violates condition (ii) of Theorem~\ref{t23}.

It is easy to construct a representing tree $T_X$ with all leaves at the same level such that a removal of one leaf violates this property.
For example, it suffices to consider a perfect strictly binary tree $T_X$ with $|X|\geqslant 4$.

Let $X$ be an ultrametric space such that all  labels of inner nodes of $T_X$ being at the same level are equal. A removal of only one leaf of $T_X$ can lead only to two possibilities: all labels remain at their places or one label disappears. In any case, the new tree preserves this property.

Clearly, the spaces having perfect strictly $n$-ary representing trees, $n\geqslant 2$, do not have the desired property.

According to Theorem~\ref{t1.4} an ultrametric space $X$ is generated by unrooted labeled tree if and only if in the representing tree $T_X$ every inner node has at least one direct successor which is a leaf. In order to show that such spaces do not have the desired property it suffices to consider a space $X$ such that in the representing tree $T_X$ there exists at least one inner node with exactly one successor which is a leaf.

Let us summarize the above considerations.  The classes of ultrametric spaces such that for every space $X$ from the given class every subspace of $X$ also belongs to this class:
\begin{itemize}
  \item The class $\mathfrak U$.
  \item The class of finite ultrametric spaces $X$ for which all labels of different internal nodes of $T_X$ are different.
  \item The class of finite ultrametric spaces $X$ for which $T_X$ is strictly binary.
  \item The class $\mathfrak R$.
  \item The class $\tilde{\mathfrak R}$.
  \item The class of finite ultrametric spaces admitting ball-preserving mappings.
  \item The class of finite ultrametric spaces $X$ for which all labels of inner nodes of  $T_X$ being at the same level are equal.
\end{itemize}
The classes which does not fulfil this condition:
\begin{itemize}
  \item The class of finite ultrametric spaces $X$ for which $T_X$ is strictly $n$-ary, $n\geqslant 3$.
  \item The class $\mct$.
  \item The class $\mathcal{TSI}$.
  \item The class of finite homogeneous ultrametric spaces.
  \item The class of finite ultrametric spaces $X$ for which all leaves of $T_X$ are on the same level.
  \item The class of finite ultrametric spaces having perfect strictly $n$-ary trees.
  \item The class of finite ultrametric spaces generated by unrooted labeled trees.
\end{itemize}

\bigskip

CONTACT INFORMATION

\medskip
Evgeniy Aleksandrovych Petrov \\

Institute of Applied Mathematics and Mechanics of the NAS of Ukraine, Slavyansk, Ukraine \\
E-Mail: eugeniy.petrov@gmail.com

\begin{thebibliography}{10}

\bibitem{WGA}
Lemin, A.~J. (2001).
\newblock {On {G}elgfand's problem concerning graphs, lattices, and ultrametric spaces}.
\newblock In {\em AAA62 - 62nd Workshop on General Algebra}, pages 12--13, Linz, Austria, June 14-17 2001.
\newblock (The abstract - http://at.yorku.ca/c/a/g/l/22.htm).

\bibitem{GurVyal(2012)}
Gurvich, V., Vyalyi, M. (2012).
\newblock {Characterizing (quasi-)ultrametric finite spaces in terms of (directed) graphs}.
\newblock {\em {Discrete Appl. Math.}, 160}(12), 1742--1756.

\bibitem{PD(UMB)}
Petrov, E., Dovgoshey, A. (2014).
\newblock On the {G}omory-{H}u inequality.
\newblock {\em J. Math. Sci., 198}(4), 392--411.
\newblock Translation from \emph{Ukr. Mat. Visn., 10}(4), 469--496, 2013.

\bibitem{DPT}
Dovgoshey, O., Petrov, E., Teichert, H.--M.  (2015).
\newblock On spaces extremal for the {G}omory-{H}u inequality.
\newblock {\em p-Adic Numbers Ultrametric Anal. Appl., 7}(2), 133--142.

\bibitem{DPT(Howrigid)}
Dovgoshey, O., Petrov, E., Teichert, H.--M.  (2017).
\newblock {How rigid the finite ultrametric spaces can be?}
\newblock {\em Fixed Point Theory Appl., 19}(2), 1083--1102.

\bibitem{DP20}
Dovgoshey, O., Petrov, E. (2020).
\newblock On some extremal properties of finite ultrametric spaces.
\newblock {\em p-Adic Numbers Ultrametric Anal. Appl., 12}(1), 1--11.

\bibitem{DP19}
Dovgoshey, O., Petrov, E. (2019).
\newblock {Properties and morphisms of finite ultrametric spaces and their
  representing trees}.
\newblock {\em {\(p\)-Adic Numbers Ultrametric Anal. Appl.}, 11}(1), 1--20.

\bibitem{DP18}
Dovgoshey, O., Petrov, E. (2018).
\newblock {From isomorphic rooted trees to isometric ultrametric spaces}.
\newblock {\em {\(p\)-Adic Numbers Ultrametric Anal. Appl.} 10}(4), 287--298.

\bibitem{Do19}
Dovgoshey, O. (2019).
\newblock {Finite ultrametric balls}.
\newblock {\em {\(p\)-Adic Numbers Ultrametric Anal. Appl.}, 11}(3), 177--191.

\bibitem{Do20BBMS}
Dovgoshey, O. (2020).
\newblock {Combinatorial properties of ultrametrics and generalized
  ultrametrics}.
\newblock {\em {Bull. Belg. Math. Soc. - Simon Stevin}, 27}(3), 379--417.

\bibitem{Do20}
Dovgoshey, O. (2020).
\newblock {Isomorphism of trees and isometry of ultrametric spaces}.
\newblock {\em {Theory Appl. Graphs}, 7}(2),
\newblock Id/No 3, 1--39

\bibitem{P18}
Petrov, E. (2018).
\newblock Weak similarities of finite ultrametric and semimetric spaces.
\newblock {\em p-Adic Numbers Ultrametric Anal. Appl., 10}(2), 108--117.

\bibitem{BM}
Bondy, J.~A., Murty, U.~S.~R. (2008).
\newblock {\em Graph theory}.
\newblock  {Graduate Texts in Mathematics}, volume 244.
\newblock Springer, New York.

\bibitem{Di}
Diestel, R. (2005).
\newblock {\em {Graph Theory}}. 
\newblock  {Graduate Texts in Mathematics}, volume 173.
\newblock Springer, Berlin.

\bibitem{DDP(P-adic)}
Dordovskyi, D., Dovgoshey, O., Petrov, E. (2011).
\newblock Diameter and diametrical pairs of points in ultrametric spaces.
\newblock {\em p-Adic Numbers Ultrametric Anal. Appl., 3}(4), 253--262.

\bibitem{GV}
Gurvich, V.,  Vyalyi, M. (2012).
\newblock Ultrametrics, trees, and bottleneck arcs.
\newblock {\em Math. Ed., 3}(16), 75--88.

\bibitem{GNS00}
Grigorchuk, R.~I., Nekrashevich, V.~V., Sushanskyi, V.~I. (2000).
\newblock {Automata, dynamical systems, and groups}.
\newblock {\em Proc. Steklov Inst. Math., 231}, 128--203.

\bibitem{H04}
Hughes, B. (2004).
\newblock Trees and ultrametric spaces: a categorical equivalence.
\newblock {\em Adv. Math., 189}(1), 148--191.

\bibitem{Le03}
Lemin, A.~J. (2003).
\newblock The category of ultrametric spaces is isomorphic to the category of complete, atomic, tree-like, and real graduated lattices {LAT*}.
\newblock {\em Algebra Universalis, 50}(1), 35--49.

\bibitem{Ho01}
{Holly}, J.~E. (2001).
\newblock {Pictures of ultrametric spaces, the \(p\)-adic numbers, and valued  fields.}
\newblock {\em {Am. Math. Mon.}, 108}(8), 721--728.

\bibitem{BS17}
Beyrer, J., Schroeder, V. (2017).
\newblock {Trees and ultrametric M\"obius structures}.
\newblock {\em {\(p\)-Adic Numbers Ultrametric Anal. Appl.}, 9}(4), 247--256,
  2017.

\bibitem{P(TIAMM)}
Petrov, E.~A. (2013).
\newblock Ball-preserving mappings of finite ulrametric spaces.
\newblock {\em Tr. Inst. Prikl. Mat. Mekh., 26}, 150--158, 
\newblock (In Russian).

\bibitem{Fr54}
{Fra\"{\i}ss\'e}, R. (1954).
\newblock {Sur l'extension aux r\'elations de quelques propri\'et\'es des ordres}.
\newblock {\em {Ann. Sci. \'Ec. Norm. Sup\'er. (3)}, 71}, 363--388.

\bibitem{Jo60}
{Jonsson}, B. (1960).
\newblock {Homogeneous universal relational systems}.
\newblock {\em {Math. Scand.}, 8}, 137--142.

\bibitem{Fr00}
{Fra\"{\i}ss\'e}, R. (2000).
\newblock {\em Theory of relations.} 
\newblock Stud. Logic Found. Math., volume 145.
\newblock  North-Holland, Amsterdam.

\bibitem{DLPS07}
{Delhomm\'e}, C.,  {Laflamme}, C.,  {Pouzet}, M.,  {Sauer}, N. (2007).
\newblock {Divisibility of countable metric spaces}.
\newblock {\em {Eur. J. Comb.}, 28}(6), 1746--1769.

\bibitem{DLPS08}
{Delhomm\'e}, C.,  {Laflamme}, C.,  {Pouzet}, M.,  {Sauer}, N. (2008).
\newblock {Indivisible ultrametric spaces}.
\newblock {\em {Topology Appl.}, 155}(14), 1462--1478.

\bibitem{DLPS16}
{Delhomm\'e}, C.,  {Laflamme}, C.,  {Pouzet}, M.,  {Sauer}, N. (2016).
\newblock On homogeneous ultrametric spaces.
\newblock https://arxiv.org/abs/1509.04346.

\bibitem{DADS}
Black, P.~E. (2020).
\newblock Perfect $k$-ary tree.
\newblock In {\em Dictionary of Algorithms and Data Structures}.
\newblock https://xlinux.nist.gov/dads/HTML/perfectKaryTree.html.

\bibitem{DK21}
Dovgoshey, O., K\"{u}c\"{u}kaslan, M. (2022).
\newblock Labeled trees generating complete, compact, and discrete ultrametric spaces.
\newblock  {\em Ann. Comb.},
\newblock  https://doi.org/10.1007/s00026-022-00581-8.

\bibitem{GomoryHu(1961)}
Gomory, R.~E.,  Hu, T.~C. (1961).
\newblock Multi-terminal network flows.
\newblock {\em SIAM, 9}(4), 551--570.

\bibitem{DP2}
Dovgoshey, O.,  Petrov, E. (2013).
\newblock Weak similarities of metric and semimetric spaces.
\newblock {\em Acta Math. Hungar., 141}(4), 301--319.
\end{thebibliography}
\end{document}